\documentclass[a4paper,11pt]{article} 

\usepackage{times}
\usepackage{latexsym}
\usepackage{amsmath}
\usepackage{amsfonts}
\usepackage{amssymb}
\usepackage{xcolor}
\usepackage{tikz}
\usetikzlibrary{shadings}
\usepackage{cite}
\usepackage[margin=10pt,font={small,stretch=0.5},labelfont=bf,labelsep=endash]{caption}
\usepackage{parskip}
\usepackage{amsthm}
\usepackage{setspace}
\usepackage{enumitem}
\usepackage[top=1in, bottom=1in, left=1in, right=1in]{geometry}

\newtheorem{Definition}[equation]{Definition}
\newtheorem{Theorem}[equation]{Theorem}
\newtheorem{Proposition}[equation]{Proposition}
\newtheorem{Lemma}[equation]{Lemma}
\newtheorem{Corollary}[equation]{Corollary}
\newtheorem{Remark}[equation]{Remark}
\newtheorem{Example}[equation]{Example}

\newcommand{\pres}[2]{\langle #1 \mid #2 \rangle}

\DeclareMathOperator{\Aut}{Aut}

\newcommand{\squarea}[3]{
\filldraw #1 circle (#2*0.02cm);
\filldraw #1+(#2,0) circle (#2*0.02cm);
\filldraw #1+(0,#2) circle (#2*0.02cm);
\filldraw #1+(#2,#2) circle (#2*0.02cm);
\draw[red] #1 -- ++(#2,0) ++(0,#2) -- ++(-#2,0);
\draw[blue] #1 ++(#2,0) -- ++(0,#2) ++(-#2,0) -- ++(0,-#2);
\draw[green!80!black] #1 -- ++(-#3,-#3);
\draw[green!80!black] #1 ++(#2,0) -- ++(#3,-#3);
\draw[green!80!black] #1 ++(0,#2) -- ++(-#3,#3);
\draw[green!80!black] #1 ++(#2,#2) -- ++(#3,#3);
}
\newcommand{\squareb}[3]{
\filldraw #1 circle (#2*0.02cm);
\filldraw #1+(#2,0) circle (#2*0.02cm);
\filldraw #1+(0,#2) circle (#2*0.02cm);
\filldraw #1+(#2,#2) circle (#2*0.02cm);
\draw[blue] #1 -- ++(#2,0) ++(0,#2) -- ++(-#2,0);
\draw[red] #1 ++(#2,0) -- ++(0,#2) ++(-#2,0) -- ++(0,-#2);
\draw[green!80!black] #1 -- ++(-#3,-#3);
\draw[green!80!black] #1 ++(#2,0) -- ++(#3,-#3);
\draw[green!80!black] #1 ++(0,#2) -- ++(-#3,#3);
\draw[green!80!black] #1 ++(#2,#2) -- ++(#3,#3);
}

\makeatletter
\def\thm@space@setup{
	\thm@preskip=\parskip \thm@postskip=0pt
}
\makeatother

\title{Automorphisms of geometric structures associated to Coxeter groups}
\author{Graham White}
\begin{document}
\maketitle

\begin{abstract}
In this paper, we consider the automorphism groups of the Cayley graph with respect to the Coxeter generators and the Davis complex of an arbitrary Coxeter group. We determine for which Coxeter groups these automorphism groups are discrete. In the case where they are discrete, we express them as semidirect products of two obvious families of automorphisms. This extends a result of Haglund and Paulin.
\end{abstract}

\section{Introduction}

Given a Coxeter system $(W,S)$, $S = \{s_i\}_{i \in I}$, with the order of each $s_is_j$ being $m_{ij} \in \{2, 3, 4, \dots, \} \cup \{\infty\}$, the corresponding \emph{defining diagram} is the graph with vertex set $\{s_i\}_{i \in I}$, and for any two vertices $s_i$ and $s_j$, an edge joining $s_i$ and $s_j$ if and only if $m_{ij}$ is finite. We label the edge between $s_i$ and $s_j$ with $m_{ij}$. As in \cite{HP}, the defining diagram is said to be \emph{flexible} if there is a vertex $s \in S$ and a nontrivial label-preserving automorphism $\phi$ of the diagram such that $\phi$ fixes $s$ and fixes each vertex connected to $s$ by an edge.

Let $\Gamma = \Gamma(W,S)$ be the Cayley graph of $W$ with respect to the generating set $S$. A \emph{left-multiplication automorphism} is an automorphism $L_w$ of $\Gamma(W,S)$ given by $L_w(x) = wx$, for some $w \in W$ and each $x \in W$. A \emph{diagram automorphism} is an automorphism of $\Gamma(W,S)$ induced by an automorphism of the Coxeter diagram for $(W,S)$. More detailed background on these geometric structures and automorphisms is presented in Section \ref{sec:background}. The topology on the group $\Aut(\Gamma)$ is described in \cite{Survey}. 

Our main result is the following theorem.

\begin{Theorem}
\label{the:main}
Let $\Gamma$ be the Cayley graph of a Coxeter system $(W,S)$ with respect to the generating set $S$, where $S$ is finite. 
\begin{enumerate}[label=\alph{enumi}),ref=\ref{the:main}\alph{enumi})]
\item \label{the:maingen} Any element of $\Aut(\Gamma)$ is the composition of a left-multiplication automorphism and a diagram automorphism if and only if the defining diagram of $(W,S)$ is not flexible.
\item \label{the:maindiscrete} $\Aut(\Gamma)$ is a discrete group if and only if the defining diagram of $(W,S)$ is not flexible.
\end{enumerate}
\end{Theorem}

\begin{Corollary}
\label{cor:mainsemidirect}
The defining diagram of the Coxeter system $(W,S)$ is not flexible if and only if the automorphism group $\Aut(\Gamma(W,S))$ can be written as a semi-direct product $$\Aut(\Gamma) = W \rtimes D$$ where $W$ and $D$ are the groups of left-multiplication and diagram automorphisms of $\Gamma$ respectively.
\end{Corollary}

We also show that Theorem \ref{the:main} and Corollary \ref{cor:mainsemidirect} remain true if $\Aut(\Gamma)$ is replaced by the automorphism group of the Davis complex for $(W,S)$. For the definition and properties of the Davis complex, see \cite{Davis}. 

In the case where the defining diagram is not flexible, and thus the automorphism group is discrete, these results were obtained by Haglund and Paulin as Theorem 5.12 of \cite{HP}, although our proof is different. Note that the defining diagrams of finite and affine Coxeter groups are not flexible. In the case that the defining diagram is flexible, \cite{HP} treated only word-hyperbolic Coxeter groups, while our method proves the result for all such Coxeter groups. An example of a Coxeter group which has flexible defining diagram and is not word-hyperbolic is $(D_{\infty} \times D_{\infty}) \ast D_{\infty}$. A discussion of word-hyperbolicity for Coxeter groups is given in Chapter 12 of \cite{Davis}.

Our proof of Theorem \ref{the:main} appears in Section \ref{sec:proofs}. We will prove Theorem \ref{the:maingen} by considering the cycles in the Cayley graph which correspond to the defining relators of the Coxeter system, and showing that automorphisms of the Cayley graph preserve these cycles. This will allow us to show that when the defining diagram is not flexible, automorphisms of the Cayley graph permute the edge labels in the same way at each vertex, and hence must be compositions of diagram automorphisms and left-multiplication automorphisms. When the defining diagram is flexible, we will explicitly extend a suitable nontrivial automorphism of the defining diagram to an automorphism of the Cayley graph which is not a composition of left-multiplication and diagram automorphisms. An infinite collection of such automorphisms which fix the identity vertex shows that the automorphism group is nondiscrete in this case. The remaining direction of Theorem \ref{the:maindiscrete} together with Corollary \ref{cor:mainsemidirect} will follow from intermediate results. 

\subsection*{Acknowledgements}

I would like to thank my supervisor, Anne Thomas, for her invaluable guidance and advice throughout this project. I would also like to thank Bob Howlett and Geordie Williamson for useful discussions.

\section{Background}
\label{sec:background}

In this section, we recall several definitions and results concerning Coxeter groups which will be used in our proof. We mostly follow Davis \cite{Davis}. In Section \ref{sec:cox}, we define Coxeter systems and state an important result about reduced words. In Section \ref{sec:cay} we define the Cayley graph and discuss its automorphisms. 

\subsection{Coxeter systems, defining diagrams and flexibility}
\label{sec:cox}

\label{def:cox}
A group $W$ is a \emph{Coxeter group} if it has a presentation of the form 
\begin{equation}
\label{eqn:coxpres}
W = \pres{\{s_i\}_{i \in I}}{\{s_i^2\}_{i \in I} \cup \{(s_is_j)^{m_{ij}}\}_{i,j \in I}} \text{, where } m_{ij} \in \{2,3, \dots\} \cup \{\infty\}
\end{equation}
and $I$ is a finite indexing set. Such a presentation is a \emph{Coxeter presentation}. Here, when $m_{ij}$ is equal to $\infty$, the product $s_is_j$ has infinite order, in which case we will omit $(s_is_j)^{m_{ij}}$ from our list of relators. The $(s_is_j)^{m_{ij}}$ for $m_{ij}$ finite are \emph{defining relators}. If $W = \pres{S}{R}$ is a Coxeter presentation for $W$, then $W$ together with its generating set $S$ is a \emph{Coxeter system}, $(W,S)$.

Given a Coxeter group $W = \pres{S}{R}$, an $m$\emph{--operation} on a word in $S$ is, for any $s,t \in S$ with $m_{st}$ finite, replacing a substring $stst\cdots$ of length $m_{st}$ by the string $tsts\cdots$ of length $m_{st}$. Theorem \ref{thm:wordprob} gives a solution to the word problem for Coxeter groups.

\begin{Theorem}[Tits, see Theorem 3.4.2 in Davis \cite{Davis}]
\label{thm:wordprob}
Given a Coxeter system $(W,S)$:
\begin{enumerate}
\item A word $\mathbf{w}$ in $S$ is reduced if and only if it cannot be shortened by a sequence of $m$--operations and/or deletion of subwords $ss, s \in S$.
\item Two reduced words in $S$ define the same group element in $W$ if and only if one can be transformed into the other by a sequence of $m$--operations.
\end{enumerate}
\end{Theorem}

\begin{Corollary}
\label{cor:samelength}
If there are multiple reduced words defining the same group element $w \in W$, then they all have the same length.
\end{Corollary}

The Coxeter diagram is the standard way of encoding a Coxeter system. However, it will be more convenient to use the alternate convention described in the introduction. There is no standard name for this other convention, so we have used `defining diagram'. 

An example of a flexible Coxeter group, that we will follow throughout this paper, is as follows:

\begin{Example}
\label{ex:part0}
Consider the Coxeter group $$W =  C_2 \ast (C_2 \times C_2) = \pres{s, t, u}{s^2, t^2, u^2,(tu)^2}.$$ The defining diagram of $W$, shown in Figure \ref{fig:flexiblediag} above, is flexible because the automorphism $\phi$ which interchanges $t$ and $u$ is a nontrivial automorphism which fixes $s$ and each vertex connected to $s$ by an edge.
\end{Example}

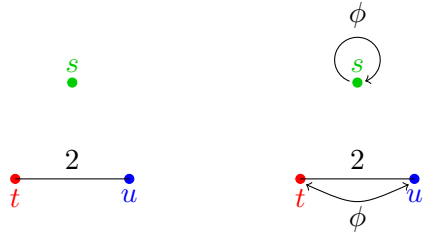
\begin{figure}
\begin{center}
\begin{tikzpicture}[scale=1.5]
\begin{scope}
\filldraw[red] (1.5,0) circle (0.04cm);
\filldraw[blue] (2.5,0) circle (0.04cm);
\filldraw[green!80!black] (2,0.8516) circle (0.04cm);
\draw (1.5,0) -- (2.5,0);
\node[above] at (2,0) {$2$};  
\node[below,red] at (1.5,0) {$t$};
\node[below,blue] at (2.5,0) {$u$};
\node[above,green!80!black] at (2,0.8516) {$s$};
\end{scope}
\begin{scope}[xshift=2.5cm]
\filldraw[red] (1.5,0) circle (0.04cm);
\filldraw[blue] (2.5,0) circle (0.04cm);
\filldraw[green!80!black] (2,0.8516) circle (0.04cm);
\draw (1.5,0) -- (2.5,0);
\node[above] at (2,0) {$2$};  
\node[below,red] at (1.5,0) {$t$};
\node[below,blue] at (2.5,0) {$u$};
\node[above,green!80!black] at (2,0.8516) {$s$};
\draw[<->] (1.55,-0.05) .. controls (2,-0.25) .. (2.45, -0.05);
\draw[->, radius=0.2cm] (2,0.8516)++(0,0.2)++(250:0.2) arc[start angle=250, end angle=-70];
\node[below] at (2, -0.15) {$\phi$};
\node[above] at (2, 1.25) {$\phi$};
\end{scope}
\end{tikzpicture}
\end{center}
\caption{The defining diagram of the Coxeter group $W = (C_2 \times C_2) \ast C_2 = \pres{s, t, u}{s^2, t^2, u^2, (tu)^2}$ and an automorphism $\phi$ illustrating that this defining diagram is flexible. Vertices are coloured to match the colours used in Figure \ref{fig:flexibleauto}.}
\label{fig:flexiblediag}
\end{figure}

The two cases we wish to consider will turn out to be exactly those Coxeter groups whose defining diagrams are or are not flexible.

\subsection{Cayley graph automorphisms}
\label{sec:cay}

\label{def:Cayley}
Let $(W,S)$ be a Coxeter system. The \emph{Cayley graph of \textit{(W,S)}} with respect to the generating set $S$, denoted $\Gamma(W,S)$, is the undirected graph constructed as follows: 
\begin{enumerate} 
\item The vertex set of $\Gamma(W,S)$ is $W$. 
\item There is a single edge of $\Gamma(W,S)$ between $w$ and $ws$, for each $w \in W$ and $s \in S$. This edge is labelled by $s$.
\end{enumerate}

Cayley graphs are usually defined as directed graphs, but because $S$ consists of elements of order two, this is not necessary here. 

We do not require that automorphisms of $\Gamma(W,S)$ preserve the edge labels. There are two straightforward families of automorphisms which we will consider: left-multiplication automorphisms and diagram automorphisms, as defined in the introduction. The following results are straightforward from definitions.

\begin{Lemma}
\label{lem:fixlabels}
Any left-multiplication automorphism $L_w$ of $\Gamma(W,S)$ preserves edge labels. That is, the image under $L_w$ of an edge labelled by $s$ is also labelled by $s$, for any $s \in S$. 
\end{Lemma}
\begin{Lemma}
\label{lem:permutelabels}
Any diagram automorphism $\Phi$ of $\Gamma(W,S)$ permutes the set of edge labels. That is, if the image under $\Phi$ of an edge labelled by $s_i$ is labelled by $s_j$, then the image under $\Phi$ of any edge labelled by $s_i$ is labelled by $s_j$. 
\end{Lemma}

Consider any automorphism $\Phi$ of $\Gamma(W,S)$ and any $w \in W$. The vertices $w$ and $\Phi(w)$ are both adjacent to $|S|$ edges, one labelled by each element of $S$. The automorphism $\Phi$ takes each edge adjacent to $w$ to an edge adjacent to $\Phi(w)$, so induces a permutation $\phi_w$ of $S$. We call the permutation $\phi_w$ the \emph{local permutation at w induced by $\Phi$}. The following corollary is immediate from Lemmas \ref{lem:fixlabels} and \ref{lem:permutelabels} and definitions.

\begin{Corollary}
\label{cor:samelabels}
Any left-multiplication automorphism of $\Gamma(W,S)$ induces the identity permutation at each vertex. Any diagram automorphism of $\Gamma(W,S)$ induces the same local permutation at each vertex.
\end{Corollary}

\section{Proofs}

\label{sec:proofs}

We now prove Theorem \ref{the:main} and Corollary \ref{cor:mainsemidirect}. In Sections \ref{sec:essential} and \ref{sec:permutations} we will prove Theorem \ref{the:maingen} in the case where the defining diagram is not flexible. We will prove the other direction of Theorem \ref{the:maingen} in Section \ref{sec:flexible}. In Section \ref{sec:discreteness}, we will prove one direction of Theorem \ref{the:maindiscrete}. The rest of Theorem \ref{the:maindiscrete} follows easily from intermediate results in Section \ref{sec:permutations}. In Section \ref{sec:semidirect}, we will show how Corollary \ref{cor:mainsemidirect} follows from Theorem \ref{the:maingen}. Finally, in Section \ref{sec:Davis} we shall extend these results to the automorphism groups of Davis complexes. 

Throughout this section, $(W,S)$ is a Coxeter system with $S$ finite and $\Gamma = \Gamma(W,S)$ is the Cayley graph of $W$ with respect to the generating set $S$. Note that the claim that any automorphism of $\Gamma$ is the composition of a diagram automorphism and a left-multiplication automorphism is equivalent to the claim that any automorphism of $\Gamma$ which fixes the vertex corresponding to the identity is a diagram automorphism. 

\subsection{Essential cycles}
\label{sec:essential}

We first observe:

\begin{Lemma} 
\label{lem:evenlength}
The graph $\Gamma$ has no cycles of odd length.
\end{Lemma}
\begin{proof}
This result follows from each of the defining relators $(st)^{m_{st}}$ having even length.
\end{proof}

Given an embedded cycle in $\Gamma$ of even length $2n$, we define a pair of vertices of that cycle to be \emph{opposite vertices} if they are distance $n$ apart in the cycle.

\begin{Definition}
\label{def:essential}
An embedded cycle of vertices $(v_1,v_2,\dotsc, v_{2n})$ of even length $2n$ in $\Gamma$ is \emph{essential} if for any two opposite vertices $v_i$ and $v_{i+n}$ of the cycle, the only paths of length $n$ from one to the other are the two contained in the cycle, namely $(v_i,v_{i+1},\dotsc, v_{i+n-1},v_{i+n})$ and $(v_i,v_{i-1},\dotsc,v_{i+n+1},v_{i+n})$, and there are no paths from $v_i$ to $v_{i+n}$ of length less than $n$.
\end{Definition}

From Lemma $\ref{lem:evenlength}$, the restriction in this definition to cycles of even length is natural.

\begin{Remark}
\label{rem:graphprop}
Note that the definition of an essential cycle depends only on graph properties. In particular, it does not depend on the group elements labelling the vertices or the generators labelling the edges of $\Gamma(W,S)$. 
\end{Remark}

Now, we will characterise the essential cycles in $\Gamma$.

\begin{Proposition}
\label{prop:essential}
A cycle in $\Gamma$ is essential if and only if it corresponds to a defining relator $(st)^{m_{st}}$ for some $s,t \in S$, with $m_{st} < \infty$.
\end{Proposition}
\begin{proof}
Consider a cycle of length $2n$ in $\Gamma$ which corresponds to a relator $(st)^{m_{st}}$, and two opposite vertices $w$ and $w'$ in this cycle. The claim that this cycle is essential is equivalent to there being exactly two reduced words in $S$ defining the group element $ww'^{-1}$, both of length $n$. 

Two reduced words defining the element $w(w')^{-1}$ are $$\mathbf{a} = \underbrace{stst\dots}_{m_{st} \text{ terms}} \text{ and } \mathbf{b} = \underbrace{tsts\dots}_{m_{st} \text{ terms}}.$$ These words are reduced by Theorem \ref{thm:wordprob}, so there are no paths between $w$ and $w'$ of length less than $n$. If there was another reduced word defining the same group element, then it could only be obtained from $\mathbf{a}$ by $m$--operations, by Theorem \ref{thm:wordprob}. However, there is only one $m$--operation which can be applied to $\mathbf{a}$, and it produces $\mathbf{b}$. Likewise, the only $m$--operation which can be applied to $\mathbf{b}$ produces $\mathbf{a}$. Thus these are the only two reduced words defining the group element $w(w'^{-1})$, so the cycle corresponding to any defining relator $(st)^{m_{st}}$ is essential, as claimed.

Now, consider an arbitrary essential cycle $$(w, \quad ws_1, \quad ws_1s_2,\quad \dots, \quad ws_1s_2\cdots s_{2n})$$

Consider the opposite vertices $w$ and $w' = ws_1s_2\cdots s_{n}$. There are no paths from $w'$ to $w$ of length less than $n$, because the cycle is essential. Since the cycle is essential, there are exactly two reduced words defining the group element $ww'^{-1}$. These are $\mathbf{a} = s_1s_2\cdots s_n$ and $\mathbf{b} = s_{2n}s_{2n-1}\cdots s_{n+1}$. Thus by Theorem \ref{thm:wordprob}, it is possible to transform $\mathbf{a}$ into $\mathbf{b}$ using $m$--operations. Any intermediate word will be a third reduced word defining $ww'^{-1}$, which would be a contradiction. Thus the application of a single $m$--operation must transform $\mathbf{a}$ into $\mathbf{b}$.

Now, $s_1$ and $s_{2n}$ are distinct, as are $s_n$ and $s_{n+1}$, because essential cycles are embedded cycles. Thus the two words $\mathbf{a}$ and $\mathbf{b}$ have different first elements and different last elements. Therefore the single $m$--operation transforming $\mathbf{a}$ into $\mathbf{b}$ must replace both the first and last element of $\mathbf{a}$, so must replace the entire word $\mathbf{a}$. Therefore, $\mathbf{a}$ is the word $\underbrace{stst\dots}_{m_{st} \text{ terms}}$ and $\mathbf{b}$ is the word $\underbrace{tsts\dots}_{m_{st} \text{ terms}}$, so our essential cycle corresponds to the defining relator $(st)^{m_{st}}$.

Thus the only essential cycles are those corresponding to defining relators, as claimed.
\end{proof}

\begin{Corollary}
\label{cor:essentialpreserved}
Any automorphism of $\Gamma$ takes a cycle corresponding to a defining relator $(st)^{m_{st}}$ to a cycle corresponding to a (potentially different) defining relator.
\end{Corollary}
\begin{proof}
As noted in Remark \ref{rem:graphprop} above, the definition of an essential cycle (Definition \ref{def:essential}) uses only graph properties. These properties are preserved by graph automorphisms, so the image of an essential cycle is an essential cycle. We know from Proposition \ref{prop:essential} that essential cycles are exactly those corresponding to defining relators. Thus the image under any graph automorphism of a cycle corresponding to a defining relator is a cycle corresponding to a defining relator. Hence, graph automorphisms preserve the set of cycles corresponding to defining relators. \end{proof}

\begin{Corollary}
\label{cor:essentialalternate}
Essential cycles in $\Gamma$ have edge labels alternating between two elements of $S$.
\end{Corollary}

\subsection{Local permutations}
\label{sec:permutations}

\begin{Proposition}
\label{prop:diagramstar}
For any automorphism $\Phi$ of $\Gamma$ which fixes the vertex corresponding to the identity, the local permutation induced by $\Phi$ at the identity induces an automorphism of the Coxeter diagram of $(W,S)$.
\end{Proposition}
\begin{proof}
For any $s,t \in S$, $m_{st} \neq \infty$ if and only if there is a unique essential cycle through the vertices $s, 1, t$ in $\Gamma$, by Proposition \ref{prop:essential}. This cycle has length $2m_{st}$. Graph automorphisms preserve essential cycles and lengths of cycles, so the images of essential cycles are essential cycles of the same length, and so $m_{st} = m_{\Phi(s)\Phi(t)}$ for any $s,t \in S$.

Thus $\Phi$ induces an automorphism of the Coxeter diagram of $(W,S)$, as required.
\end{proof}

The statement of the following lemma is illustrated by Figure \ref{fig:samelabels}.

\begin{Lemma}
\label{lem:samelabels}
Consider two elements $s,t \in S$ with $m_{st}$ finite and two adjacent vertices $w$ and $ws$ in $\Gamma$. These vertices are joined by an edge $e$, which is labelled by $s$. Consider the two edges $e_1 = \{w,wt\}$ and $e_2 = \{ws,wst\}$. For any automorphism $\Phi$ of $\Gamma$, the images of the two edges $e_1$ and $e_2$ have the same label.
\end{Lemma}
\begin{proof}
Because $m_{st}$ is finite, there is an essential cycle containing the path $$(wt, e_1, w, e, ws, e_2, wst)$$ from the vertex $wt$ to $wst$. The edge labels along this path are $t$, $s$, $t$, as illustrated in Figure $\ref{fig:samelabels}$. From Corollary \ref{cor:essentialpreserved}, we know that the image of this essential cycle is an essential cycle. Hence, by Corollary \ref{cor:essentialalternate} the image of this essential cycle has alternating edge labels, so the images of $e_1$ and $e_2$ have the same label, as claimed.
\end{proof}

\begin{figure}
\begin{center}
\begin{tikzpicture}[scale=1.5,xscale=2.5]
\begin{scope}
\draw[blue] (-0.9,0.3) -- (0,0) (1,0) -- (1.9,0.3);
\draw[red] (0,0) -- (1,0);
\node[above,blue] at (-0.45,0.15) {$t$};
\node[above,blue] at (1.45,0.15) {$t$};
\node[above,red] at (0.5,0) {$s$};
\node[below] at (-0.45,0.15) {$e_1$};
\node[below] at (1.45,0.15) {$e_2$};
\node[below] at (0.5,0) {$e$};
\node[below] at (-0.9,0.3) {$wt$};
\node[below] at (0,0) {$w$};
\node[below] at (1,0) {$ws$};
\node[below] at (1.9,0.3) {$wst$};
\end{scope}
\draw[->] (0.5,-0.23) -- (0.5,-0.73);
\node[right] at (0.5,-0.5) {$\phi$};
\begin{scope}[yshift=-1cm]
\draw[blue] (-0.9,0.3) -- (0,0) (1,0) -- (1.9,0.3);
\draw[red] (0,0) -- (1,0);
\node[above,blue] at (-0.45,0.15) {$t'$};
\node[above,blue] at (1.45,0.15) {$t'$};
\node[above,red] at (0.5,0) {$s'$};
\node[below] at (-0.45,0.15) {$\phi(e_1$)};
\node[below] at (1.45,0.15) {$\phi(e_2)$};
\node[below] at (0.5,0) {$\phi(e)$};
\node[below] at (-0.9,0.3) {$\phi(wt)$};
\node[below] at (0,0) {$\phi(w)$};
\node[below] at (1,0) {$\phi(ws)$};
\node[below] at (1.9,0.3) {$\phi(wst)$};
\end{scope}
\end{tikzpicture}
\end{center}
\caption{The path discussed in Lemma \ref{lem:samelabels} and its image. There is an essential cycle through the path $\phi(e_1)\phi(e)\phi(e_2)$, so the edges $\phi(e_1)$ and $\phi(e_2)$ have the same label.}
\label{fig:samelabels}
\end{figure}
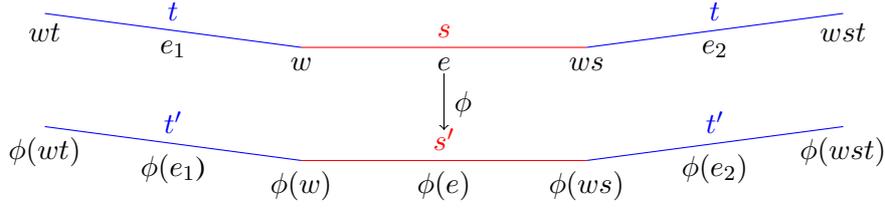

\begin{Corollary}
\label{cor:adjsame}
Suppose the defining diagram of $(W,S)$ is not flexible. Consider two adjacent vertices $w$ and $ws$ of $\Gamma$. For an arbitrary automorphism $\Pi$ of $\Gamma$, let the local permutations of the edge labels at $w$ and $ws$ induced by $\Pi$ be $\pi_w$ and $\pi_{ws}$ respectively. Then the permutation $\pi_{ws}^{-1}\pi_w$ of $S$ induces an automorphism of the Coxeter diagram which fixes both the vertex $s$ and the vertices corresponding to any $t \in S$ with $m_{st}$ finite. 
\end{Corollary}
\begin{proof}
Recall that $L_w$ is the automorphism of $\Gamma$ defined by left-multiplication by $w$. Consider the automorphism $$\Psi = (L_{\Pi^{-1}(ws)})^{-1} \circ \Pi^{-1} \circ L_{ws} \circ (L_{\Pi(w)})^{-1} \circ \Pi \circ L_w.$$ 

By the construction of $\Psi$, it is not difficult to check that $\Psi$ fixes the identity vertex of $\Gamma$.

The local permutation induced by $\Psi$ at the identity is $\pi_{ws}^{-1}\pi_w$, because the four left-multiplication automorphisms in the definition of $\Psi$ preserve the edge labels. Thus we may apply Proposition \ref{prop:diagramstar}, and deduce that $\pi_{ws}^{-1}\pi_w$ induces an automorphism of the Coxeter diagram of $(W,S)$.

Now, $\pi_w(s) = \pi_{ws}(s)$, because the edge between $w$ and $ws$ is labelled by $s$. Therefore, $\pi_{ws}^{-1}\pi_w(s) = s$. From Lemma \ref{lem:samelabels}, we know that for any $t \in S$ with $m_{st}$ finite, $\pi_w(t) = \pi_{ws}(t)$. Therefore $\pi_{ws}^{-1}\pi_w(t) = t$ for any $t \in S$ with $m_{st}$ finite. Therefore, the permutation $\pi_{ws}^{-1}\pi_w$ induces an automorphism of the Coxeter diagram which fixes both $s$ and any $t \in S$ with $m_{st}$ finite, as required.
\end{proof}

\begin{Corollary}
\label{cor:labelpermute}
If the defining diagram of $(W,S)$ is not flexible, then any automorphism $\Pi$ of $\Gamma$ permutes the edge labels.
\end{Corollary}
\begin{proof}
From Corollary \ref{cor:adjsame}, we have that if the local permutations induced by $\Pi$ at any two adjacent vertices of $\Gamma$ are $\pi_1$ and $\pi_2$, then $\pi_2^{-1}\pi_1$ induces an automorphism of the Coxeter diagram, and thus an automorphism of the defining diagram. We also have that this induced automorphism of the defining diagram fixes both $s$ and each $t \in S$ which is connected to $s$ by an edge (in the defining diagram), using the definition of the defining diagram. If the defining diagram is not flexible, then the automorphism induced by $\pi_2^{-1}\pi_1$ must be the trivial automorphism, so $\pi_1 = \pi_2$. Thus the local permutations at any two adjacent vertices must be the same. Because $\Gamma$ is connected, it follows that the local permutations induced by $\Pi$ are the same at any two vertices of $\Gamma$. Therefore $\Pi$ permutes the edge labels of $\Gamma$. 
\end{proof}

All that remains to prove Theorem \ref{the:maingen} for Coxeter groups whose defining diagrams are not flexible is the following.

\begin{Proposition}
\label{prop:isdiagram}
If the defining diagram of $(W,S)$ is not flexible, then any automorphism $\Phi$ of $\Gamma$ which permutes the edge labels and fixes the identity element is a diagram automorphism.
\end{Proposition}
\begin{proof}
Let the local permutation induced by $\Phi$ at the identity be $\phi_e$. The permutation $\phi_e$ induces an automorphism of the Coxeter diagram, by Proposition \ref{prop:diagramstar}. This induced automorphism of the Coxeter diagram in turn induces a diagram automorphism $\Phi_e$ of $\Gamma$. The automorphism $\Phi\Phi_e^{-1}$ fixes the identity vertex and fixes the edge labels around every vertex, so since $\Gamma$ is connected, it is the identity map. Hence $\Phi = \Phi_e$, so $\Phi_e$ is a diagram automorphism, as required.
\end{proof}

Thus for any Coxeter system $(W,S)$ whose defining diagram is not flexible, any automorphism of $\Gamma$ is the composition of a left-multiplication and a diagram automorphism. We have established one direction of Theorem \ref{the:maingen}.

\subsection{Flexible defining diagrams}
\label{sec:flexible}

Throughout this section we assume that the Coxeter system $(W,S)$ has defining diagram which is flexible. By Corollary \ref{cor:samelabels}, to show the remaining direction of Theorem \ref{the:maingen}, it suffices to construct an automorphism of $\Gamma$ which takes two edges with the same label to two edges with different labels. 

Our construction is as follows. Since the defining diagram of $(W,S)$ is flexible, we may choose $s \in S$ and a nontrivial automorphism $\phi$ of the defining diagram of $(W,S)$ which fixes $s \in S$ and also fixes each $t \in S$ which is connected to $s$ by an edge in the defining diagram. In particular, $\phi$ is a permutation of $S$ which fixes $s$ and each $t$ with $m_{st}$ finite. Throughout this section, $\phi$ and $s$ will be fixed. By abuse of notation, we denote by $\phi$ the map on words over $S$ which applies $\phi$ to each letter.

We will first define a function $\Psi_\phi$ on reduced words in $S$, and then show that this function induces a well-defined automorphism of $\Gamma(W,S)$, which we shall also denote by $\Psi_\phi$. 

\begin{Definition}
\label{def:interestingaut}
Let $\mathbf{w}$ be an arbitrary reduced word in $S$. If $\mathbf{w}$ contains $s$, then write $\mathbf{w} = w_1sw_2$, where the (possibly empty) subword $w_1$ does not contain $s$, and define $\Psi_\phi(\mathbf{w}) = \phi(w_1)sw_2$. If $\mathbf{w}$ does not contain $s$, then define $\Psi_\phi(w) = \phi(w)$. That is, 

$$
\Psi_\phi(\mathbf{w}) = \left\{ \begin{array}{lcl} \phi(w_1)sw_2 & \text{if} & \mathbf{w} = w_1sw_2 \\ \phi(\mathbf{w}) & \text{if} & s \text{ is not in } \mathbf{w}. \end{array} \right.
$$
\end{Definition}

\begin{Example}
\label{ex:part1}
Consider the Coxeter group and the automorphism $\phi$ from Example \ref{ex:part0}. In this example, $\Psi_\phi$ acts on a reduced word $\mathbf{w}$ by interchanging any $t,u$ which occur before the first instance of $s$ in $\mathbf{w}$, and fixing the remainder of the word.
\end{Example}

To simplify notation, we will use $\Psi$ instead of $\Psi_\phi$. In order to show that $\Psi$ induces a well-defined automorphism of $\Gamma(W,S)$, we will use Theorem \ref{thm:wordprob} to show that both deleting subwords of the form $tt, t \in S$, and $m$--operations commute with the operation of applying $\phi$, in the sense of the following proposition.

\begin{Proposition}
\label{prop:comdel}
For any word $\mathbf{w}$ in $S$:
\begin{enumerate}[label=(\arabic{enumi})] 
\item If a word $\mathbf{w'}$ is obtained from $\mathbf{w}$ by deleting a subword of the form $tt$ with $t \in S$ and then applying $\phi$, then $\mathbf{w'}$ can also be obtained by first applying $\phi$ to $\mathbf{w}$ and then deleting a subword of the form $\phi(t)\phi(t)$ from $\phi(\mathbf{w})$.
\item If a word $\mathbf{w''}$ is obtained by applying an $m$--operation $m_1$ to $\mathbf{w}$ and then applying $\phi$, then $\mathbf{w''}$ can also be obtained by applying $\phi$ to $\mathbf{w}$ and then applying a (potentially different) $m$--operation $m_2$ to $\phi(\mathbf{w})$.
\end{enumerate}
\end{Proposition}
\begin{proof}
For (1), let $d_1$ be the operation of deleting an instance of $tt$ from $\mathbf{w}$. Let $d_2$ be the operation of deleting $\phi(t)\phi(t)$ from the corresponding position in $\phi(\mathbf{w})$. Because the $n$th letter of $\phi(\mathbf{w})$ depends only on the $n$th letter of $\mathbf{w}$, the word $(\phi \circ d_1)(\mathbf{w})$ is equal to the word $(d_2 \circ \phi)(\mathbf{w})$. The proof of (2) is similar.
\end{proof} 

\begin{Corollary}
\label{cor:trans}
A word $\mathbf{w}$ in $S$ can be transformed into another word $\mathbf{w'}$ in $S$ by a sequence of $m$--operations and deletions of subwords $tt$ with $t \in S$ if and only if the word $\phi(\mathbf{w})$ in $S$ can be transformed into the word $\phi(\mathbf{w'})$ in $S$ by another sequence of $m$--operations and deletions of subwords $tt$ with $t \in S$.
\end{Corollary}

We now deduce some properties of the transformation $\Psi$.

\begin{Proposition}
\label{prop:alphareduced}
For any reduced word $\mathbf{w} = s_1s_2 \cdots s_n$, the image $\Psi(\mathbf{w})$ is a reduced word.
\end{Proposition}
\begin{proof}
If the reduced word $\mathbf{w}$ does not contain $s$, then by definition $\Psi(\mathbf{w}) = \phi(\mathbf{w})$. If $\Psi(\mathbf{w})$ is not a reduced word, then there is a sequence of $m$--operations and deletions of repeated pairs which transforms $\Psi(\mathbf{w}) = \phi(\mathbf{w})$ into a shorter word, which we may denote $\phi(\mathbf{w'})$ for some word $\mathbf{w'}$, as $\phi$ is a bijection on the set of words in $S$. By Corollary \ref{cor:trans}, there is a sequence of operations transforming $\mathbf{w}$ into the word $\mathbf{w'}$. Now, $\mathbf{w'}$ is shorter than $\mathbf{w}$, because $\phi$ preserves the lengths of words. Hence $\mathbf{w}$ is not a reduced word, which is a contradiction.

In the case where $\mathbf{w}$ does contain $s$, our argument is very similar. Any $m$--operation involving $s$ must act on strings comprised of $s$ and $t$ with $m_{st}$ finite, and any such $t$ is fixed by $\phi$. Thus any $m$--operation acts entirely before the first $s$, entirely after the first $s$, or involves only letters fixed by $\phi$. Therefore any $m$--operation commutes with $\phi$ in the sense of Proposition \ref{prop:comdel}.

By definition, we have $\Psi(\mathbf{w}) = \phi(w_1)sw_2$. If $\Psi(\mathbf{w})$ is not a reduced word, then there is a sequence of operations which transforms $\Psi(\mathbf{w})$ into a shorter word. As with the previous case, Corollary \ref{cor:trans} implies that there is a sequence of operations transforming $\mathbf{w}$ into a shorter word, contradicting the fact that $\mathbf{w}$ is a reduced word. The only difference from the previous case is that there are $m$--operations acting on a section of the word unaffected by $\phi$, and these will be included in the new sequence unchanged, rather than being transformed according to Corollary \ref{cor:trans}. 

Therefore in either case, $\Psi(\mathbf{w})$ is a reduced word. 
\end{proof}

\begin{Proposition}
\label{prop:alphaaut}
The map $\Psi$ on reduced words in Definition \ref{def:interestingaut} above induces a well-defined automorphism of $\Gamma$. 
\end{Proposition}
\begin{proof}
Firstly, we show that $\Psi$ induces a well defined map on the vertex set $W$ of $\Gamma$. That is, that given two reduced words $\mathbf{w}$ and $\mathbf{w'}$ in $S$ which define the same group element in $W$, $\Psi(\mathbf{w})$ and $\Psi(\mathbf{w'})$ must define the same group element. Since $\mathbf{w}$ and $\mathbf{w'}$ are reduced words, $\Psi(\mathbf{w})$ and $\Psi(\mathbf{w'})$ are reduced words, by Proposition \ref{prop:alphareduced}. By Theorem \ref{thm:wordprob}, it suffices to show that if there is a sequence of $m$--operations transforming $\mathbf{w}$ into $\mathbf{w'}$ then there is a sequence of $m$--operations transforming $\Psi(\mathbf{w})$ into $\Psi(\mathbf{w'})$. This result follows from Corollary \ref{cor:trans} in the same way that Proposition \ref{prop:alphareduced} does. Hence, $\Psi$ is a well-defined map on $W$. Note also that $\Psi$ is a bijection, because $(\Psi_{\phi})^{-1} = \Psi_{\phi^{-1}}$.

All that remains is to show that $\Psi$ induces an automorphism of $\Gamma$. To do this, we need to check that vertices $w$ and $w'$ are adjacent if and only if $\Psi(w)$ and $\Psi(w')$ are adjacent. This is immediate from the definition of $\Psi$.
\end{proof}

\begin{Example}
Continuing from Example \ref{ex:part1}, Figure \ref{fig:flexibleauto} shows part of the Cayley graph $\Gamma(W,S)$, along with the automorphism $\Psi$. The automorphism $\Psi$ can be thought of as interchanging some `branches' of $\Gamma$ while leaving others fixed.
\end{Example}

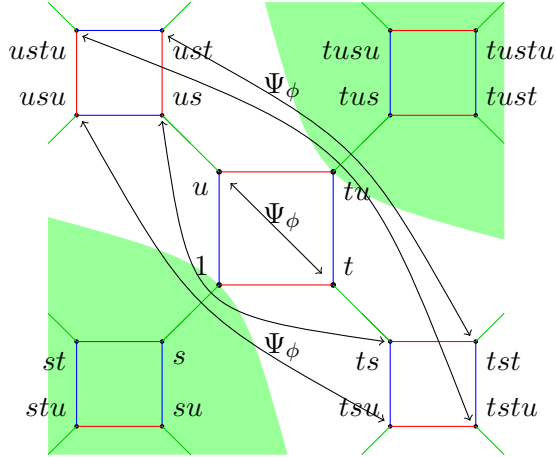
\begin{figure}
\begin{center}
\begin{tikzpicture}[scale=1.5]
\fill[green!40] (0.4,2.5) .. controls (0.85,0.85) .. (2.5,0.4) -- (2.5,2.5) -- (0.4,2.5);
\fill[green!40] (0.6,-1.5) .. controls (0.15,0.15) .. (-1.5,0.6) -- (-1.5,-1.5) -- (0.6,-1.5);

\squarea{(0,0)}{1}{0.5}

\squarea{(1.5,1.5)}{0.75}{0.25}
\squareb{(-1.25,1.5)}{0.75}{0.25}
\squarea{(1.5,-1.25)}{0.75}{0.25}
\squarea{(-1.25,-1.25)}{0.75}{0.25}

\node[above left] at (0,0) {$1$};
\node[below left] at (0,1) {$u$};
\node[above right] at (1,0) {$t$};
\node[below right] at (1,1) {$tu$};

\node[above left] at (-1.25,1.5) {$usu$};
\node[below left] at (-1.25,2.25) {$ustu$};
\node[above right] at (-0.5,1.5) {$us$};
\node[below right] at (-0.5,2.25) {$ust$};

\node[above left] at (-1.25,-1.25) {$stu$};
\node[below left] at (-1.25,-0.5) {$st$};
\node[above right] at (-0.5,-1.25) {$su$};
\node[below right] at (-0.5,-0.5) {$s$};

\node[above left] at (1.5,1.5) {$tus$};
\node[below left] at (1.5,2.25) {$tusu$};
\node[above right] at (2.25,1.5) {$tust$};
\node[below right] at (2.25,2.25) {$tustu$};

\node[above left] at (1.5,-1.25) {$tsu$};
\node[below left] at (1.5,-0.5) {$ts$};
\node[above right] at (2.25,-1.25) {$tstu$};
\node[below right] at (2.25,-0.5) {$tst$};

\draw[<->] (0.1,0.9) -- (0.9,0.1);
\draw[<->] (-1.2,1.45) .. controls (-0.25,-0.25) .. (1.45,-1.2);
\draw[<->] (-0.45,2.2) .. controls (1.25,1.25) .. (2.2,-0.45);
\draw[<->] (-0.5,1.45) .. controls (-0.25,-0.25) .. (1.45,-0.5);
\draw[<->] (-1.2,2.2) .. controls (1.25,1.25) .. (2.2,-1.2);

\node at (0.55,1.75) {$\Psi_\phi$};
\node at (0.55,0.6) {$\Psi_\phi$};
\node at (0.55,-0.55) {$\Psi_\phi$};
\end{tikzpicture}
\end{center}

\caption{Part of the Cayley graph of the Coxeter group from \ref{ex:part1}, and the automorphism $\Psi_\phi$ constructed from the automorphism $\phi$ illustrated in Figure \ref{fig:flexiblediag}. The automorphism $\Psi_\phi$ fixes the areas shaded green, and interchanges the other two `branches' as shown.}
\label{fig:flexibleauto}
\end{figure}

We will now show that the automorphism $\Psi$ is not the composition of diagram and left-multiplication automorphisms.

\begin{Proposition}
\label{prop:notpermute}
The automorphism $\Psi$ does not permute the edge labels of $\Gamma$.
\end{Proposition}
\begin{proof}
Choose an element $t \in S$ with $\phi(t) \neq t$. This is possible because $\phi$ is a nontrivial automorphism of the Coxeter diagram of $(W,S)$.

There is an edge $e_1$ labelled by $t$ between the vertices $1$ and $t$. The image $\Psi(e_1)$ is the edge between $\Psi(1) = 1$ and $\Psi(t) = \phi(t)$, which is labelled by $\phi(t) \neq t$. There is another edge, $e_2$, labelled by $t$ between the vertices $s$ and $st$. The image $\Psi(e_2)$ is the edge between $\Psi(s)=s$ and $\Psi(st)=st$, which is also labelled by $t$. These edges are shown in Figure \ref{fig:alpha}.

Hence the automorphism $\Psi$ takes two edges which are both labelled by $t$, $e_1$ and $e_2$, to edges labelled by $\phi(t) \neq t$ and $t$ respectively, which means that $\Psi$ does not permute the edge labels of $\Gamma(W,S)$.
\end{proof} 

\begin{figure}
\begin{center}
\begin{tikzpicture}[scale=1.8]
\begin{scope}
\filldraw (0,0) circle(0.02cm);
\filldraw (90:1) circle(0.02cm);
\filldraw (210:1) circle(0.02cm);
\filldraw (330:1) circle(0.02cm);
\filldraw (210:1)++(90:1) circle(0.02cm);
\draw[red] (0,0) -- (210:1);
\draw[blue] (0,0) -- (90:1);
\draw[blue] (210:1) -- ++(90:1);
\draw[green!80!black] (0,0) -- (330:1);
\node[below] at (0,0) {$1$};
\node[above] at (90:1) {$t$};
\node[below] at (330:1) {$\phi(t)$};
\node[below] at (210:1) {$s$};
\draw (210:1)++(90:1.15) node {$st$};
\node[blue,right] at (90:0.5) {$t$};
\draw[blue] (210:1)++(90:0.5)++(0.1,0) node {$t$};
\node[above,green!80!black] at (330:0.6) {$\phi(t)$};
\node[below,red] at (210:0.4) {$s$};
\end{scope}
\begin{scope}[xshift=3cm]
\filldraw (0,0) circle(0.02cm);
\filldraw (90:1) circle(0.02cm);
\filldraw (210:1) circle(0.02cm);
\filldraw (330:1) circle(0.02cm);
\filldraw (210:1)++(90:1) circle(0.02cm);
\draw[red] (0,0) -- (210:1);
\draw[blue] (0,0) -- (90:1);
\draw[blue] (210:1) -- ++(90:1);
\draw[green!80!black] (0,0) -- (330:1);
\node[below] at (0,0) {$1$};
\node[above] at (90:1) {$t$};
\node[below] at (330:1) {$\phi(t)$};
\node[below] at (210:1) {$s$};
\draw (210:1)++(90:1.15) node {$st$};
\node[blue,right] at (90:0.5) {$t$};
\draw[blue] (210:1)++(90:0.5)++(0.1,0) node {$t$};
\node[above,green!80!black] at (330:0.6) {$\phi(t)$};
\node[below,red] at (210:0.4) {$s$};
\draw[radius=1cm, ->] (80:1) arc[start angle=80, end angle=-20];
\draw[radius=0.5cm, ->] (60:0.5) arc[start angle=60, end angle=0];
\draw[radius=0.2cm, ->] (210:1.2)++(-30:0.2) arc[start angle=330, end angle=60];
\draw[radius=0.2cm, ->] (210:1)++(90:1.2)++(240:0.2) arc[start angle=240, end angle=-60];
\draw[radius=0.2cm, ->] (210:1)++(90:0.5)++(0.1,0.1) arc[start angle=150, end angle=-150];
\node at (30:0.6) {$\phi$};
\node at (30:1.1) {$\phi$};
\node at (-0.32,0.05) {$\phi$};
\node at (210:1.55) {$\phi$};
\node at (-0.78,1.05) {$\phi$};
\end{scope}
\end{tikzpicture}
\end{center}
\caption{The part of $\Gamma(W,S)$ used in the proof of Proposition \ref{prop:notpermute}, and the action of $\phi$. Note that $\Psi$ takes two edges labelled by $t$ to edges labelled by $t$ and $\phi(t) \neq t$.}
\label{fig:alpha}
\end{figure}
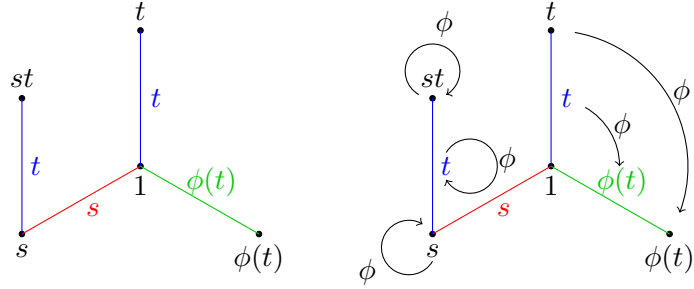

\begin{Corollary}
The automorphism $\Psi$ is not a composition of left-multiplication and diagram automorphisms.
\end{Corollary}

Hence, for any Coxeter system $(W,S)$ whose defining diagram is flexible, we have constructed an automorphism of $\Gamma$ which is not the composition of left-multiplication and diagram automorphisms. This completes the proof of Theorem \ref{the:maingen}.

\subsection{Nondiscreteness of the automorphism group}
\label{sec:discreteness}

In the previous section, we considered an automorphism $\phi$ of the (flexible) defining diagram which fixes $s \in S$ and each $t \in S$ with $m_{st}$ finite. From this, we constructed an automorphism $\Psi$ of $\Gamma(W,S)$ by applying $\phi$ to each letter before the first $s$ in an arbitrary reduced word. We will now construct an infinite family of automorphisms $\Psi_1, \Psi_2, \dotsc, \Psi_n, \dotsc$ which fix the identity in order to show that $\Aut(\Gamma)$ is nondiscrete. As discussed in \cite{Survey}, the automorphism group is discrete only if the stabiliser of the identity is finite.

For $n = 1, 2, \dotsc$, let $\Psi_n$ be defined similarly to $\Psi$, except that rather than changing some letters at the beginning of the word and fixing the remainder, $\Psi_n$ leaves every letter before the $n$th $s$ unchanged, and applies $\phi$ to the remainder. If there are fewer than $n$ occurences of $s$, then $\Psi_n$ fixes the word. The proof of the following proposition is almost identical to that of Proposition \ref{prop:alphaaut}. 

\begin{Proposition}
The $\Psi_n$ are well defined automorphisms of $\Gamma(W,S)$.
\end{Proposition} 

Take any element $t \in S$ such that $\phi(t) \neq t$. Then $m_{st} = \infty$, so $st, (st)^2, (st)^3, \dotsc$ are all distinct. For each $n$, $\Psi_n$ fixes exactly those $(st)^k$ with $k \leq n$, so no two of the $\Psi_n$ are equal. Thus we have constructed an infinite family of automorphisms of $\Gamma(W,S)$ which fix the identity vertex. Hence this vertex has an infinite stabiliser, and so the automorphism group is nondiscrete when the defining diagram is flexible. This proves Theorem \ref{the:maindiscrete} in the case where the defining diagram is flexible. 

If the defining diagram is not flexible, then by Corollary \ref{cor:labelpermute} and Proposition \ref{prop:isdiagram}, any automorphism of $\Gamma(W,S)$ which fixes the identity vertex induces an automorphism of the defining diagram. The defining diagram is finite, so has finitely many automorphisms. Thus the stabiliser of the identity vertex is finite, so the automorphism group of $\Gamma(W,S)$ is discrete, completing the proof of Theorem \ref{the:maindiscrete}.

\subsection{The automorphism group as a semidirect product}
\label{sec:semidirect}

Given a Coxeter system $(W,S)$ whose defining diagram is flexible, Theorem \ref{the:maingen} gives us that any element of $\Aut(\Gamma)$ can be written as the composition of an element of $L$ and an element of $D$, where $L$ and $D$ are the subgroups of left-multiplication and diagram automorphisms of $\Aut(\Gamma)$, respectively. Note that $L$ is isomorphic to $W$.

The automorphism group of $\Gamma(W,S)$ always has $W \rtimes D$ as a subgroup. Theorem \ref{the:maingen} thus implies that $\Aut(\Gamma) \cong W \rtimes D$ exactly when the defining diagram is not flexible, proving Corollary \ref{cor:mainsemidirect}.

\subsection{The automorphism group of the Davis complex}
\label{sec:Davis}

In this section, we prove the following proposition, which implies that Theorem $\ref{the:main}$ and Corollary $\ref{cor:mainsemidirect}$ are true of the Davis complex, as well as of the Cayley graph. 

\begin{Proposition}
\label{prop:autiso}
Let $\Sigma = \Sigma(W,S)$ be the Davis complex of a Coxeter system $(W,S)$. Then the automorphism groups $\Aut(\Gamma)$ and $\Aut(\Sigma)$ are isomorphic.
\end{Proposition}

We prove this proposition in a sequence of lemmas, showing how each of the Cayley graph and the Davis complex can be constructed from the other. 

Firstly, note that from the Davis complex $\Sigma$, we can construct a graph $\Gamma'$ as follows. For each chamber of $\Sigma$, $\Gamma'$ has a vertex. For each pair of chambers of $\Sigma$ which are joined along the mirror of type $s \in S$, the graph $\Gamma'$ has an edge between the corresponding two vertices, which is labelled by $s$. This graph $\Gamma'$ is just the Cayley graph $\Gamma$, from the definition of $\Sigma$. Likewise, we could construct the Davis complex $\Sigma(W,S)$ from the Cayley graph $\Gamma$, by taking a chamber for each vertex of $\Gamma$ and joining the the chambers corresponding to the vertices $w_1$ and $w_2$ along the mirror corresponding to $s \in S$ if and only if $w_1$ and $w_2$ are joined by an edge labelled by $s$ in $\Gamma$. 

\begin{Lemma}
\label{lem:sigmacay}
Any automorphism of $\Sigma$ induces an automorphism of $\Gamma$. 
\end{Lemma}
\begin{proof}
An automorphism of $\Sigma$ must take chambers to chambers and preserve adjacency, so induces a mapping on the vertices of $\Gamma$ which preserves adjacency, which are the conditions required to be a graph automorphism. Thus any automorphism of $\Sigma$ induces an automorphism of $\Gamma$.
\end{proof}

\begin{Lemma}
\label{lem:chamberauto}
The automorphism groups of the chamber $K(W,S)$ and the Coxeter diagram of $(W,S)$ are isomorphic.
\end{Lemma}
\begin{proof}
A permutation $\phi$ of $S$ induces an automorphism of both $K(W,S)$ and of the Coxeter diagram of $(W,S)$ exactly when it satisfies $m_{st} = m_{\phi(s)\phi(t)}$ for all $s,t \in S$. From this, the lemma follows.
\end{proof}

\begin{Lemma}
\label{lem:caysigma}
Any automorphism of $\Gamma$ induces an automorphism of $\Sigma$. 
\end{Lemma}
\begin{proof}
Any automorphism of $\Gamma$ induces a map from $\Sigma$ to itself which takes chambers to chambers. Together with Lemma \ref{lem:chamberauto}, this implies the result.  
\end{proof}

It is easily verified that the maps of Lemmas \ref{lem:sigmacay} and \ref{lem:caysigma} are mutually inverse bijective homomorphisms, so are isomorphisms. Thus $\Aut(\Gamma)$ and $\Aut(\Sigma)$ are isomorphic, proving Proposition \ref{prop:autiso}.

Defining the automorphisms of $\Sigma(W,S)$ induced by left-multiplication and diagram automorphisms of $\Gamma$ as left-multiplication and diagram automorphisms of $\Sigma(W,S)$, respectively, Proposition \ref{prop:autiso} implies that Theorem $\ref{the:main}$ and Corollary $\ref{cor:mainsemidirect}$ are true of the Davis complex, as well as of the Cayley graph.

\bibliography{Bibliography}
\bibliographystyle{plain}

\end{document}